\documentclass[11pt]{article}
\usepackage[T1]{fontenc}
\usepackage{lmodern,amsmath,amssymb,amsthm,mathtools}
\usepackage[a4paper,margin=28mm]{geometry}
\usepackage{microtype,amscd}
\usepackage{indentfirst}
\usepackage[colorlinks=true,allcolors=blue,hyperfootnotes=false]{hyperref}
\newtheorem{theorem}{Theorem}[section]
\newtheorem{proposition}[theorem]{Proposition}
\newtheorem{lemma}[theorem]{Lemma}

\theoremstyle{definition}
\theoremstyle{remark}
\newcommand{\Z}{\mathbb Z}\newcommand{\Q}{\mathbb Q}
\newcommand{\C}{\mathbb C}\newcommand{\R}{\mathbb R}\newcommand{\N}{\mathbb Z_{\geq0}}
\newcommand{\A}{\mathbb A}\newcommand{\Gm}{\mathbb G_m}
\newcommand{\GL}{\operatorname{GL}}\newcommand{\Sym}{\operatorname{Sym}}
\newcommand{\Hom}{\operatorname{Hom}}\newcommand{\Ad}{\operatorname{Ad}}
\newcommand{\sgn}{\operatorname{sgn}}\newcommand{\IH}{\operatorname{IH}}
\newcommand{\supp}{\operatorname{supp}}
\newcommand{\ssmod}{\mathrm{ss}}
\newcommand{\qbin}[3]{\begin{bmatrix}#1\\#2\end{bmatrix}_{#3}}
\newcommand{\pos}{\Z_{\geq0}[v^{\pm1}]}
\numberwithin{equation}{section}
\title{A Note on Universal Positivity\\of Rank-Two Quantum Greedy Elements}
\author{Qiyue Tang}
\date{26 September 2026}
\begin{document}
\maketitle
\begingroup
\renewcommand{\thefootnote}{}
\footnotetext{\textit{2020 Mathematics Subject Classification.}
Primary 13F60; Secondary 16G20, 17B37.\newline
\textit{Key words and phrases.} Quantum cluster algebra, greedy basis,
universal positivity, quiver moduli, scattering diagram.}
\endgroup
\begin{abstract}
We prove universal positivity of rank-two quantum greedy elements when
$b\mid c$ or $c\mid b$.
\end{abstract}

\section{Introduction}
An element of a rank-two cluster algebra has a Laurent expansion in
each cluster. Universal positivity requires nonnegative coefficients in
all of these expansions. We study this property for the quantum greedy basis.

The greedy basis selects elements indexed by integer denominator labels.
Its defining recurrence is accompanied by a support characterization:
the prescribed leading monomial, a restricted support region, and
divisibility conditions from the neighboring clusters determine an element
uniquely. Lee--Li--Rupel--Zelevinsky constructed the quantum greedy basis
\cite[Theorems~1.7 and 1.9]{LLRZ}. Its coefficients lie in
$\Z[v^{\pm1}]$. Specializing $v=1$ can conceal negative coefficients,
and the alternating sums in the greedy recurrence make quantum
positivity a separate question from existence.

We consider the divisibility range $b\mid c$ or $c\mid b$ proposed in
\cite[Conjecture~1.10]{LLRZ}. When $c=kb$, one exchange factor splits into $k$ factors with a common
quantum parameter. Schur positivity in the resulting source colors
gives positive quantum differences after these factors are recombined
on the original integer lattice.

\subsection{Quantum greedy elements}

Let $b,c$ be positive integers. We use the skew field generated over $\Q(v)$
by $X_1,X_2$, subject to $X_2X_1=v^2X_1X_2$. Put
\begin{equation}\label{eq:torus}
 X^{(r,s)}=v^{rs}X_1^rX_2^s,
 \qquad X^mX^n=v^{-\det(m,n)}X^{m+n}.
\end{equation}
The parameter $v$ is indeterminate. The coefficient-free quantum cluster
algebra is the $\Z[v^{\pm1}]$-subalgebra of this skew field generated by
the variables $X_m$, $m\in\Z$. The ordered pairs $(X_m,X_{m+1})$ are
its clusters, and the variables satisfy
\begin{equation}\label{eq:exchange}
 X_{m+1}X_{m-1}=\begin{cases}
 v^bX_m^b+1,&m\text{ odd},\\
 v^cX_m^c+1,&m\text{ even}.
 \end{cases}
\end{equation}
Each adjacent pair satisfies $X_{m+1}X_m=v^2X_mX_{m+1}$, so
$v^{rs}X_m^rX_{m+1}^s$ is the corresponding normalized Laurent monomial.
Throughout, positive means coefficientwise nonnegative, with zero allowed;
for example, $v^2+v^{-2}$ is positive but $v^2-1+v^{-2}$ is not,
although its specialization at $v=1$ is positive.

The symmetric quantum integers and binomial coefficients are denoted by
\[
 [n]_u=\frac{u^n-u^{-n}}{u-u^{-1}},\qquad
 \qbin nk u=\prod_{i=1}^k\frac{[n-k+i]_u}{[i]_u}
 \quad(0\leq k\leq n).
\]
For nonnegative $n$, the binomial is zero outside this range. Write
$[a]_+=\max(a,0)$.

For $(A,N)\in\Z^2$, the quantum greedy element is
\begin{equation}\label{eq:greedy}
 G_{bc}(A,N)=\sum_{p,q\geq0}e_{A,N}(p,q)
                  X^{(-A+bp,-N+cq)}.
\end{equation}
Here $e_{A,N}(0,0)=1$, and the remaining coefficients are characterized by
\begin{equation}\label{eq:recurrence}
e_{A,N}(p,q)=
\begin{cases}
\displaystyle\sum_{i=1}^{p}(-1)^{i-1}e_{A,N}(p-i,q)
 \qbin{[N-cq]_++i-1}{i}{v^b},&cAq\leq bNp,\\[2mm]
\displaystyle\sum_{j=1}^{q}(-1)^{j-1}e_{A,N}(p,q-j)
 \qbin{[A-bp]_++j-1}{j}{v^c},&cAq\geq bNp.
\end{cases}
\end{equation}
The two expressions agree on the boundary. Existence, finite support and
uniqueness are supplied by \cite[Theorem~1.7]{LLRZ}.

\begin{theorem}\label{thm:main}
Suppose that $b\mid c$ or $c\mid b$. For every $(A,N)\in\Z^2$ and
$m\in\Z$, all coefficients of $G_{bc}(A,N)$ in the normalized monomials
$v^{rs}X_m^rX_{m+1}^s$ belong to $\pos$.
\end{theorem}

By \cite[Theorem~1.9(d)]{LLRZ}, these elements are also indecomposable:
none is a sum of two nonzero universally positive elements.
This proves \cite[Conjecture~1.10]{LLRZ}.

\subsection{Outline of the proof}
The classical theta basis agrees with the classical greedy basis
\cite{CGMMRSW}. For $c(v)\in\pos$, the implication
$c(1)=0\Rightarrow c(v)=0$ recovers the exact classical support from
coefficientwise specialization. Applied in three adjacent charts, it
gives quantum Laurentness and hence the greedy divisibility conditions.
The uniqueness theorem \cite[Proposition~3.7]{LLRZ} identifies the lift.

The other difficulty is integrality on the original lattice. Assume
$c=kb$. Splitting one incoming quantum dilogarithm into $k$ factors gives
an integral star-quiver model for the wall coefficients. For the ray
$(-bp,cq)$, with $\gcd(p,q)=1$, the necessary quantum difference has
step $b\gcd(p,k)$, which can be smaller than the step visible in an
individual split factor.

Source-color Schur positivity supplies the required division. Here it
means expansion in Schur polynomials in the source colors with
coefficients in $\N[t^{\pm1}]$, where $t$ records cohomological degree.
Its contribution is the positive quotient
\[
 \frac{[\gcd(M,k)]_x}{[k]_x}
 s_\lambda(x^{k-1},x^{k-3},\ldots,x^{1-k})
 \in\N[x^{\pm1}],\qquad |\lambda|=M.
\]
For a multiple $n(p,q)$, take $M=np$ and use
$\gcd(np,k)\mid n\gcd(p,k)$ to obtain the smaller difference.
We construct the symmetric-group representations underlying this Schur
expansion on a smooth proper family of framed moduli. Stabilization
recovers the semistable stack, and equivariant cohomological integrality
extracts its primitive contribution, retaining the grading signs.
The resulting wall actions enter the formal scattering construction
of \cite{DM} on $\Z^2$ and produce the positive lift used above.

\medskip
\noindent\textbf{Use of generative AI.}
The main results of this paper were generated using GPT-6-Astra and the
Danus system. The author supplied relevant references, checked the
manuscript, and revised its exposition. Related work may have been
overlooked in this AI-assisted process; the author welcomes comments
and additional references.

\section{Schur specialization and positive quantum differences}
Let $Y=(Y_1,\ldots,Y_k)$ be commuting color variables. For
$a\in\N^k$, write $|a|=\sum_i a_i$ and $Y^a=\prod_iY_i^{a_i}$.
For a partition $\lambda$ with at most $k$ parts, $s_\lambda(Y)$
denotes the character of the polynomial representation $S_\lambda(\C^k)$;
equivalently,
\[
 s_\lambda(Y)=
 \frac{\det(Y_i^{\lambda_j+k-j})_{i,j=1}^k}
      {\det(Y_i^{k-j})_{i,j=1}^k}.
\]
The quotient is a symmetric polynomial of degree $|\lambda|$.
A color polynomial is Schur-positive over $\N[t^{\pm1}]$ if its
coefficients in this basis belong to that semiring. The variable $t$
records a grading; it is independent of $Y$ until a specialization is
explicitly made.

The alphabet $(x^{k-1},x^{k-3},\ldots,x^{1-k})$ is the weight alphabet
of the $k$-dimensional irreducible $\mathfrak{sl}_2$-module. It allows
us to combine divisibility with symmetry and unimodality of characters.
For instance, $s_{(1)}$ specializes to $[k]_x$, so the quotient below
equals one when $M=1$.

\begin{lemma}\label{lem:schurdivision}
Let $\lambda$ be a partition of $M>0$ with at most $k$ parts, and let
$d=\gcd(M,k)$. Then
\[
 \frac{[d]_x}{[k]_x}
 s_\lambda(x^{k-1},x^{k-3},\ldots,x^{1-k})
 \in\N[x^{\pm1}].
\]
\end{lemma}
\begin{proof}
Set $S(x)=s_\lambda(x^{k-1},\ldots,x^{1-k})$. This is the character of
the Schur module $S_\lambda(V)$, where $V$ is the $k$-dimensional
irreducible $\mathfrak{sl}_2$-module. Complete reducibility gives
\begin{equation}\label{eq:unimodal}
 S_j=S_{-j},\qquad S_j\geq S_{j-2}\quad(j\leq0),
 \qquad S_j=[x^j]S(x).
\end{equation}
All weights have parity $M(k-1)$.

Consider a root $\zeta$ of
$D(x)=[k]_x/[d]_x$ and let $h$ be the order of $\zeta^2$.
Then $h\mid k$ and $h\nmid d$, so $h\nmid M$.
Multiplication by $\zeta^{-2}$ permutes the principal alphabet.
Symmetry and homogeneity therefore give
$S(\zeta)=\zeta^{-2M}S(\zeta)$, whence $S(\zeta)=0$.
All roots of $D$ are simple. After multiplying by monomials, $D$ is
monic with integer coefficients. Thus $Q=S/D$ is an integral Laurent
polynomial.

The quotient is symmetric. Expanding at zero gives, for $j\leq0$,
\[
 [x^j]Q
 =\sum_{a\geq0}
  \bigl(S_{j-k+d-2ka}-S_{j-k-d-2ka}\bigr)\geq0.
\]
The sum is finite, and both indices lie on the nonpositive side.
Equation~\eqref{eq:unimodal} applies by iteration in steps of two.
Symmetry handles $j>0$.
\end{proof}

Define the formal quantum dilogarithm by
\begin{equation}\label{eq:psi}
 \Psi_s(Z)=\exp\left(
 \sum_{h\geq1}\frac{(-1)^{h-1}Z^h}{h(s^h-s^{-h})}\right).
\end{equation}
The exponential is taken in $\Q(s)[[Z]]$, where $Z$ has positive
formal degree. In particular, $\Psi_s(s^2Z)/\Psi_s(Z)=1+sZ$.
The following result converts a graded Schur expansion into a positive
shifted quotient of quantum dilogarithms, with the parity of the grading determining whether its Euler
factors are polynomials or geometric series.

\begin{proposition}\label{prop:difference}
Fix positive integers $p,k$ and put $g=\gcd(p,k)$.
Index the $k$ colors here by $0\leq i<k$. For every $n\geq1$, let
\[
 \Omega_n(t;Y)=\sum_{|a|=np,j}\beta_{n,a,j}t^jY^a
 =\sum_{\lambda\vdash np}c_{n,\lambda}(t)s_\lambda(Y),
 \qquad c_{n,\lambda}(t)\in\N[t^{\pm1}].
\]
Suppose that all nonzero $t$-degrees of $\Omega_n$ have the same parity
$\epsilon_n$. Define
\[
 f(Z)=\prod_{n,a,j}
 \Psi_{x^k}\left((-1)^j x^{kj+\sum_i(k-1-2i)a_i}Z^n\right)
 ^{(-1)^j\beta_{n,a,j}}.
\]
Then $f(x^{2g}Z)/f(Z)\in1+Z\N[x^{\pm1}][[Z]]$.
\end{proposition}
\begin{proof}
Put $d_n=\gcd(np,k)$. Since $\gcd(p/g,k/g)=1$,
\[
 d_n=g\gcd(np/g,k/g)=g\gcd(n,k/g)\mid ng.
\]
Lemma~\ref{lem:schurdivision} shows that
\begin{equation}\label{eq:primitive}
 P_n(x)=\frac{[ng]_x}{[k]_x}
 \Omega_n(x^k;x^{k-1},\ldots,x^{1-k})
 =\sum_d m_{n,d}x^d
 \in\N[x^{\pm1}].
\end{equation}
Indeed, for each Schur summand one first divides by $[k]_x/[d_n]_x$
and then multiplies by the positive polynomial $[ng]_x/[d_n]_x$.

Taking logarithms in the $Z$-adic completion yields
\[
 \log\frac{f(x^gZ)}{f(x^{-g}Z)}
 =\sum_{n,h\geq1}
 \frac{(-1)^{(h-1)(\epsilon_n+1)}}{h}
 P_n(x^h)Z^{nh}.
\]
Consequently the centered quotient is a product of factors
\begin{equation}\label{eq:euler}
 \begin{cases}
 (1+x^dZ^n)^{m_{n,d}},&\epsilon_n=0,\\
 (1-x^dZ^n)^{-m_{n,d}},&\epsilon_n=1.
 \end{cases}
\end{equation}
Every coefficient is positive and finite. Replacing $Z$ by $x^gZ$
proves the assertion.
\end{proof}

\section{Source colors and ordered eigenvalues}
A color is a source vertex, and its multiplicity is the dimension
assigned to that vertex. The ordered-eigenvalue construction identifies
cohomology for source dimensions $a_1,\ldots,a_k$ with
$S_{a_1}\times\cdots\times S_{a_k}$-invariants in a representation
associated with $P$ labeled one-dimensional sources.

All varieties and stacks in this section are over $\C$; cohomology has
rational coefficients. For a quotient stack $[X/G]$, $H^j([X/G])$
means equivariant cohomology $H_G^j(X;\Q)$. For a smooth variety of
dimension $D$ we use the centered character $\sum_j\dim H^j t^{j-D}$;
for a smooth quotient stack the same convention uses its stack
dimension, which can be negative. These shifts will determine the
parity in the wall factors.

Fix positive integers $b,P,Q$. For a composition
$a=(a_1,\ldots,a_k)$ of $P$, allowing zero parts, let $\mathcal Q_a$
be the star with $k$ sources, one sink, and $b$ arrows from each source
to the sink. The dimensions are $(a_1,\ldots,a_k,Q)$. We use the semistability condition
\begin{equation}\label{eq:stability}
 Q\sum_i\dim U_i'-P\dim W'\leq0
\end{equation}
for every subrepresentation. Explicitly, a representation consists of
spaces $U_i=\C^{a_i}$ and $W=\C^Q$ with $b$ maps $U_i\to W$ for
each $i$; a subrepresentation is a choice of subspaces $U_i',W'$
preserved by all these maps. Stability requires strict inequality for every nonzero proper
subrepresentation. The acting change-of-basis group is
$\prod_i\GL(U_i)\times\GL(W)$.
Write $\mathfrak S_a$ for the quotient of the semistable locus by this
group, retaining the common scalar stabilizer, and put
\[
 \delta_a=bPQ-Q^2-\sum_i a_i^2,
 \qquad \delta_{\mathrm{fl}}=bPQ-Q^2-P.
\]

\subsection{Framed approximation}
We approximate the stack by projective moduli with additional framing
vectors. Increasing the number of vectors raises the codimension of
unstable framings and recovers each fixed cohomological degree.

Add a vertex of dimension one and $f$ arrows from it to each source.
For a rational number $0<\varepsilon<1$, assign weights
\[
 \theta_i=Q-\varepsilon/P,\qquad
 \theta_{\mathrm{sink}}=-P,\qquad
 \theta_{\mathrm{frame}}=\varepsilon.
\]
Let $M_a^f$ be the stable moduli space. A subrepresentation with total
source dimension $A$, sink dimension $B$ and frame dimension
$\eta\in\{0,1\}$ has weight
\begin{equation}\label{eq:perturbation}
 QA-PB+\varepsilon(\eta-A/P).
\end{equation}
We use the convention that subrepresentations have nonpositive weight.
To compare with \cite[Definition~2.19]{Hoskins}, clear denominators and
replace $\theta$ by $-\theta$.
Vanishing of \eqref{eq:perturbation} forces $(A,B,\eta)$ to be the zero
or full dimension vector. Thus semistability equals stability. Equivalently, the underlying
representation is semistable and no proper equal-slope subrepresentation
contains every framing vector. Standard quiver GIT \cite[\S2.2]{Hoskins}
therefore gives a smooth projective variety, possibly empty, of dimension
\begin{equation}\label{eq:frameddim}
 D_a^f=fP+bPQ-Q^2-\sum_i a_i^2.
\end{equation}

\begin{proposition}\label{prop:framed}
There is a finite graded $\Q[S_P]$-module $R^f$, independent of the
number of colors, such that
\begin{equation}\label{eq:framedyoung}
 \sum_j\dim H^j(M_a^f)t^{j-D_a^f}
 =\sum_d\dim(R_d^f)^{S_{a_1}\times\cdots\times S_{a_k}}t^d.
\end{equation}
\end{proposition}
\begin{proof}
Merge the sources into $U=\C^P$, let $W=\C^Q$, and add an
endomorphism $L$ of $U$. Retain the $b$ maps $U\to W$ and the $f$
framing vectors in $U$, with stability~\eqref{eq:perturbation}
for subrepresentations whose source subspaces are $L$-invariant.
The affine invariant quotient is $\Sym^P\A^1$: scaling at the sink
eliminates outgoing-arrow coordinates from invariants, scaling at the
frame eliminates the framing coordinates, and the remaining invariants
are the characteristic coefficients of $L$. Its stable moduli space
$M^f$ is consequently proper over $\Sym^P\A^1$.

Adjoin a complete $L$-invariant flag in $U$. The resulting space $F^f$
has the ordered-eigenvalue map $h_f$. Let $r_f$ forget the flag,
and let $c_f$ record the characteristic polynomial of $L$. Then
\begin{equation}\label{eq:eigenvalue-square}
\begin{CD}
 F^f @>{r_f}>> M^f\\
 @V{h_f}VV @VV{c_f}V\\
 \A^P @>{\pi}>> \Sym^P\A^1
\end{CD}
\end{equation}
commutes. An ordered list of eigenvalues determines a unique invariant
flag over the distinct-eigenvalue locus; at collisions $r_f$ has flag
fibers.

The map $h_f$ is proper: the forgetful map to $M^f$ is proper, and the graph of
$h_f$ is closed in the base change by the finite map
$\pi:\A^P\to\Sym^P\A^1$. It is smooth as well. In a flag basis,
the parameter space is an open subset of upper triangular matrices,
$b$ arbitrary arrow matrices and $f$ vectors; projection to the diagonal
entries is smooth. The group
$(B_P\times\GL_Q\times\GL_1)/\Gm$ acts freely, and smoothness
descends along its quotient torsor. The relative dimension is
\[
 D^f=fP+bPQ-Q^2-P.
\]
The local systems $E_f^j=R^jh_{f,*}\Q$ are therefore constant on
$\A^P$. If the family is empty, both sides of~\eqref{eq:framedyoung} vanish.

Over pairwise distinct ordered eigenvalues, the fiber is the framed star
with $P$ labeled one-dimensional sources. Reordering the eigenvalues
permutes those labels. These maps extend across the diagonals because
$E_f^j$ is constant and the configuration space is connected. They give
an $S_P$-representation on its common fiber.

We use the Betti Grothendieck--Springer action with the normalization
in which the top cohomology of the complete flag variety is the sign
representation. If $\rho_d$ is the Springer resolution of the nilpotent
cone and $i_0$ is the inclusion of its zero endomorphism, the sign summand is
\[
 \Hom_{S_d}(\sgn,R\rho_{d,*}\Q)
       \cong i_{0,*}\Q[-d(d-1)].
\]
At eigenvalue collisions the action restricts to the product of the
block actions. These are the block decomposition and sign-support
identities of \cite[\S3, Corollaries~3.4, 3.7 and the proof of
Proposition~3.8]{FYZ},
stated there for $\ell$-adic sheaves; here we use their Betti form,
given by the same flag correspondences over $\C$.
To specify the base change, put $\mathfrak g=\operatorname{End}(\C^P)$
and let $\widetilde{\mathfrak g}$ parameterize pairs consisting of an
endomorphism and an invariant complete flag. Let $F_{\mathrm{fr}}$ be
the framing line. The common scalar acts trivially on
$U\otimes F_{\mathrm{fr}}^\vee$, so this bundle and its endomorphism
$L$ descend along the free effective quotient. They define $\ell_f$;
their relative space of complete $L$-invariant flags is $F^f$.
Thus we have the Cartesian square of stacks
\[
\begin{CD}
 F^f @>>> [\widetilde{\mathfrak g}/\GL_P]\\
 @V{r_f}VV @VV{\rho}V\\
 M^f @>{\ell_f}>> [\mathfrak g/\GL_P].
\end{CD}
\]
The stability condition is imposed on the framed loop-quiver representation.
Since $\rho$ is proper, proper base change gives
$Rr_{f,*}\Q\cong\ell_f^*R\rho_*\Q$, with the pulled-back Springer
action. The block and sign identities can therefore be applied on each
fixed-eigenvalue fiber with its arrow and framing parameters.

The comparison with the action on $E_f^j$ takes place on
\[
 R^j(c_fr_f)_*\Q=\pi_*E_f^j.
\]
Since $E_f^j$ is constant and $\pi$ is finite,
$\pi_*E_f^j[P]$ is the intermediate extension of its restriction to
the unordered distinct-eigenvalue locus. Its endomorphisms are determined by this restriction.
The action induced by relabeling eigenvalues and the pushed-forward
Springer action agree there, hence on $\pi_*E_f^j[P]$. At a tuple of
repeated eigenvalues, restriction to its stabilizer gives the block
action used in the sign-summand calculation.

At an eigenvalue tuple with multiplicities $a_i$, let
$g_a=\sum_i a_i(a_i-1)/2$. The block-sign summand of the flag direct
image is supported where each generalized eigenspace has scalar loop
endomorphism. This locus is $M_a^f$: the arrows and framing vectors
split into their source-color components and the stability inequalities
are unchanged. The local system from the top cohomology of the flag
fibers is trivial, since the groups $\GL_{a_i}$ are connected. Thus
\begin{equation}\label{eq:springeridentity}
 \Hom_{S_{a_1}\times\cdots\times S_{a_k}}
 \left(\sgn,E_f^j\right)
 \cong H^{j-2g_a}(M_a^f).
\end{equation}
Now $D^f-D_a^f=2g_a$. Set
$R^f_{j-D^f}=E_f^j\otimes\sgn_{S_P}$. The restriction of the global
sign is the product of the block signs, so
\eqref{eq:springeridentity} gives~\eqref{eq:framedyoung}.
\end{proof}

\subsection{The semistable stack}
To pass to the semistable stack, we make the preceding construction
compatible with the transition maps in $f$.

\begin{proposition}\label{prop:stack}
There is a graded $S_P$-representation $R^\infty$, bounded below and
finite in each degree, for which
\begin{equation}\label{eq:stackyoung}
 \sum_j\dim H^j(\mathfrak S_a)t^{j-\delta_a}
 =\sum_d\dim(R_d^\infty)^{S_{a_1}\times\cdots\times S_{a_k}}t^d.
\end{equation}
It is the global sign twist of the centered cohomology of the star with
$P$ labeled one-dimensional sources, with its geometric permutation action.
\end{proposition}
\begin{proof}
Write $X_a=\operatorname{Rep}_a^{\ssmod}$ and
$G_a=\prod_i\GL_{a_i}\times\GL_Q$. The good framing locus is an
open subset $E_f^{\mathrm{good}}$ of
$E_f=X_a\times(\bigoplus_i\C^{a_i})^f$. The group $G_a$ acts freely
there, with quotient $M_a^f$. The common scalar in $G_a$ acts freely
on the nonzero framing tuples.

Put $C_0=\lfloor P^2/4\rfloor+\lfloor Q^2/4\rfloor$.
For a proper equal-slope subdimension $(\alpha_i,B)$, with
$A=\sum_i\alpha_i<P$, the choices of subspaces have dimension at most
\[
 \sum_i\alpha_i(a_i-\alpha_i)+B(Q-B)
 \leq A(P-A)+B(Q-B)\leq C_0.
\]
Requiring all $f$ vectors to belong to the chosen source sum has
codimension $f(P-A)\geq f$. The incidence image is closed because the
Grassmannians are projective. Taking the finite union over subdimensions
shows that $E_f\setminus E_f^{\mathrm{good}}$ has codimension at least
$f-C_0$. This includes the zero subrepresentation, which excludes the
all-zero framing tuple.

Localization and equivariant homotopy invariance give
\begin{equation}\label{eq:stableapprox}
 H^j(\mathfrak S_a)\xrightarrow{\sim}H^j(M_a^f)
 \qquad\bigl(j<2(f-C_0)-1\bigr).
\end{equation}
Indeed, cohomology with supports in a closed subset of complex codimension
$d$ vanishes below degree $2d$ in a smooth variety. Stratification and
the relative Thom isomorphism prove this assertion, and the Borel
spectral sequence gives the equivariant version.

Appending a zero framing vector defines $M_a^f\to M_a^{f+1}$.
The induced pullbacks agree with~\eqref{eq:stableapprox} in the stable
range. The same estimate applies uniformly to the ordered-eigenvalue
flag families, including colliding eigenvalues: fixed-diagonal upper
triangular loop matrices with the semistability condition form a smooth
open parameter space, and the same Grassmannians bound bad framings.
The pullbacks on the constant local systems $E_f^j$ commute with
permutations over distinct eigenvalues, hence everywhere. Denote their
stabilized $S_P$-representation by $E_\infty^j$.

The sign-summand identifications~\eqref{eq:springeridentity} are also
compatible with these maps, by proper base change for the same flag map.
Passing to the stable range gives
\[
 \Hom_{\prod_i S_{a_i}}(\sgn,E_\infty^j)
 =H^{j-2g_a}(\mathfrak S_a).
\]
Since $\delta_{\mathrm{fl}}-\delta_a=2g_a$, put
$R^\infty_{j-\delta_{\mathrm{fl}}}=E_\infty^j\otimes\sgn_{S_P}$.
This proves~\eqref{eq:stackyoung}. Over distinct eigenvalues the family
is the product of the configuration space with the thin-star stack.
Reordering is ordinary permutation of its source labels, which proves
the final assertion about the representation.
\end{proof}

\section{Schur positivity of BPS characters}
The primitive part of semistable cohomology determines the wall factor. In the zero-potential
acyclic setting used here, the BPS space is the centered intersection
cohomology of the coarse semistable moduli when stable objects exist,
and is zero otherwise; the interface with Hall factorization is
\cite[Proposition~6.14]{DM}. Intersection cohomology is normalized here
so that it agrees with ordinary cohomology for a smooth variety.

Fix coprime positive integers $p,q$. For $n\geq1$, put $P=np$, $Q=nq$.
Let $N_a$ be the coarse semistable star moduli for $|a|=P$ and sink
dimension $Q$. If the stable locus is nonempty, define
\begin{equation}\label{eq:bpspolynomial}
 \Omega_a(t)=\sum_j\dim\IH^j(N_a)t^{j-D_a},
 \qquad D_a=1-\sum_i a_i^2-Q^2+bPQ.
\end{equation}
Set $\Omega_a=0$ when that locus is empty. A nonempty stable locus is
dense, and the coarse moduli is projective because the star is acyclic.

The integrality theorem of \cite[Theorems~A and C]{DMi} expresses the
semistable cohomology as a graded symmetric algebra on BPS cohomology
tensored with the scalar-stabilizer factor. Its PBW form retains a
canonical map compatible with the source-permutation action. Here a graded symmetric
algebra uses the Koszul rule: exchanging degrees $i,j$ contributes
$(-1)^{ij}$. The additional sign prescribed by the Euler form is written
explicitly in the proof. Equivariance passes the source-permutation
representations from the stack to the BPS space.

\begin{lemma}\label{lem:pbw-equivariance}
Fix a finite source set $I$, with $b$ arrows from every source to the
sink, zero potential, and stability depending only on total source and
sink dimensions. At a fixed positive slope, the canonical PBW map of
\cite[Theorem~C]{DMi} commutes with permutations of $I$, acting by
relabeling dimension vectors. This compatibility preserves the centered
gradings, the scalar factor $H(B\C^*)_{\mathrm{vir}}$, and the
Euler-form symmetry used in that theorem. A summand supported on
$J\subset I$ is viewed on this same quiver by setting the other source
dimensions to zero.
\end{lemma}
\begin{proof}
A permutation of $I$ permutes source spaces and their arrow maps and
fixes the sink. It preserves semistability and commutes with
semisimplification. Pullback by this isomorphism preserves intersection
complexes, perverse truncation, and the full-support summand defining
the primitive inclusion. The zero-potential inclusion and its absolute
version are those of \cite[Corollaries~4.11 and 5.9]{DMi}.

The common scalar subgroup is carried identically to itself. Hence
the induced action on its cohomology, including its degree-two
generator, is the identity. Relabeling also carries the stack of short
exact sequences and both maps of the Hall correspondence to themselves.
Consequently it commutes with the primitive inclusion and Hall
multiplication, which construct the PBW map
\cite[Theorem~C and \S6.3]{DMi}.

For zero potential, the monodromic vanishing-cycle functor is the
identity by \cite[Example~2.11]{DMi}. Verdier duality on the smooth
semistable stack identifies the absolute cohomology in
\cite[equation~(1)]{DMi} with ordinary cohomology in degrees
$j-\dim\mathfrak S_a$. These identifications commute with source
permutations. The scalar factor has degrees $1,3,5,\ldots$ and
character $t/(1-t^2)$. All shifts and Tate normalizations depend on
dimensions and Euler pairings, which relabeling preserves; the
relabeling maps have degree zero. The Euler-form
symmetrizing sign is likewise preserved, so the resulting PBW map is
equivariant with its original twisted symmetry. Setting a source
dimension to zero removes its vector space and arrows and has the
same effect in each correspondence. This proves the compatibility
for summands with smaller source support.
\end{proof}

\begin{theorem}\label{thm:colors}
For every $n$ there is a finite graded rational $S_{np}$-representation
$R_n$, independent of $k$, such that
\[
 \Omega_a(t)=\sum_d\dim(R_{n,d})^{\prod_iS_{a_i}}t^d.
\]
Consequently $\Omega_n(t;Y)=\sum_{|a|=np}\Omega_a(t)Y^a$ is
Schur-positive over $\N[t^{\pm1}]$.
\end{theorem}
\begin{proof}
For a finite source set $I$ of cardinality $P=np$, consider the star
whose sources all have dimension one and whose sink has dimension $nq$.
Let $H_I$ and $B_I$ be its centered semistable-stack cohomology and BPS
cohomology, with the geometric action of source permutations. The latter
is the centered intersection cohomology in~\eqref{eq:bpspolynomial},
or zero under its stated convention.

We use the cohomological integrality and canonical PBW isomorphisms of
\cite[Theorems~A and C]{DMi}. At zero potential, their scalar factor is
$T=H(B\C^*)_{\mathrm{vir}}$, with character $t/(1-t^2)$.
The canonical PBW exchange has, besides the cohomological Koszul sign,
the parity
\begin{equation}\label{eq:tau}
 \tau(d_1,d_2)=\chi(d_1,d_2)
             +\chi(d_1,d_1)\chi(d_2,d_2)\pmod2.
\end{equation}
The Euler form of the star is
\[
 \chi((a,Q),(a',Q'))=\sum_i a_i a_i'+QQ'-b|a|Q'.
\]
Its antisymmetric part vanishes between dimension vectors of the fixed
projected slope $(p,q)$. Thus the generic-slope hypothesis of those
theorems is satisfied, including noncoprime multiples.

By Lemma~\ref{lem:pbw-equivariance}, this PBW isomorphism is
equivariant for source permutations, with its stated twisted symmetry.

Restrict this PBW isomorphism to dimension one at each source in $I$.
Every factor is supported on a subset $I_j$, and these subsets partition
$I$. Its source and sink dimensions are $(P_j,Q_j)=(n_jp,n_jq)$.
There are no nonzero source-free factors at this positive slope.
For two disjoint source supports,
\[
 \chi_{12}=Q_1Q_2-bP_1Q_2,\qquad
 \chi_{ii}=P_i+Q_i^2-bP_iQ_i.
\]
Reducing~\eqref{eq:tau} modulo two gives
\begin{equation}\label{eq:signcancel}
 \tau_{12}=n_1n_2\bigl(q-bpq+(p+q-bpq)^2\bigr)
           =P_1P_2\pmod2.
\end{equation}
Tensor both $H_I$ and $B_I$ with $\sgn_I$, in cohomological degree zero.
Exchanging two sign-twisted label blocks supplies $(-1)^{P_1P_2}$.
Equation~\eqref{eq:signcancel} cancels the additional PBW sign.
We obtain, as graded symmetric sequences,
\begin{equation}\label{eq:species}
 R_H=\Sym_{\mathrm{seq}}(T\otimes R_B),\qquad
 R_{H,I}=\sgn_I\otimes H_I,
 \quad R_{B,I}=\sgn_I\otimes B_I,
\end{equation}
with ordinary Koszul symmetry. Here the tensor product of symmetric
sequences sums over disjoint decompositions of a finite label set.

Evaluate such a sequence on an even $k$-dimensional color space $V$:
\[
 E_k(R)=\bigoplus_P(R_P\otimes V^{\otimes P})^{S_P}.
\]
This functor is symmetric monoidal: induction from
$S_{P_1}\times S_{P_2}$ becomes tensor product, compatibly with block
permutations. Its weight-$a$ subspace is $R_P^{\prod_iS_{a_i}}$.
Proposition~\ref{prop:stack} identifies the character of $E_k(R_H)$
with the colored semistable-stack cohomology and its
stack-dimension shift. Hence~\eqref{eq:species} gives
\[
 H_{\mathrm{colored}}
 =\Sym\bigl(T\otimes E_k(R_B)\bigr)
\]
at the level of full graded characters. The ordinary integrality
isomorphism for the colored quiver also gives
\[
 H_{\mathrm{colored}}=\Sym(T\otimes B_{\mathrm{colored}}).
\]
Interpret these character identities in
$\Q((t))[[Y_1,\ldots,Y_k]]$, completed by total source degree.
At each fixed color weight, the BPS spaces are finite-dimensional and
$T$ is bounded below and degreewise finite. Only finitely many
decompositions into positive source weights contribute, so every
coefficient belongs to $\Q((t))$. In source degree $np$, terms with
at least two factors use smaller positive source degrees and cancel
by induction. The remaining equality is $T$ times the equality of
primitive characters; multiplying by $t^{-1}-t$ proves their equality.

Take $R_n=R_{B,I}$ for $|I|=np$. The preceding equality identifies its
Young invariants with~\eqref{eq:bpspolynomial}. The Frobenius character
of a finite graded symmetric-group representation is a nonnegative
graded sum of Schur polynomials, proving the theorem.
\end{proof}

\section{Positivity of quantum wall actions}
A wall acts on a quantum torus by conjugation. If its series is
$f(Z)\in1+Z\Q(v)[[Z]]$, with $Z=X^D$, then
\[
 \Ad_{f(Z)}(X^m)=X^m
 \frac{f(v^{-2\det(D,m)}Z)}{f(Z)},
 \qquad \Ad_f(a)=faf^{-1}.
\]
Thus positivity of a wall action is determined by its shifted quotient. Along the ray $D=(-bp,cq)$,
the possible pairings on the original lattice have smallest positive
value $\gcd(bp,cq)$.

Assume $c=kb$. Put
\[
 u=(-b,0),\qquad w=(0,c),\qquad t=v^{-c},
 \qquad A_0=\Psi_{v^{-b}}(X^u),\quad B_0=\Psi_{v^{-c}}(X^w).
\]
The mixed ray factors $f_{p,q}(Z)$ are defined by the unique ordered
factorization
\begin{equation}\label{eq:rayfactor}
 B_0A_0B_0^{-1}
 =A_0\prod_{q/p\,\mathrm{increasing}}f_{p,q}(X^{pu+qw}),
 \qquad p,q>0,\quad\gcd(p,q)=1.
\end{equation}
Products are completed in nonnegative $(u,w)$-degree: modulo terms of
total degree greater than a fixed bound, only finitely many factors
contribute. Each $f_{p,q}$ has constant term one and contains all
positive multiples of its primitive degree. Ordering the slopes makes
the factorization unique, recursively in total degree.

\begin{proposition}\label{prop:wall}
Let $g=\gcd(p,k)$ and $Z=X^{pu+qw}$. Then
\begin{equation}\label{eq:minimal}
 \frac{f_{p,q}(v^{-2bg}Z)}{f_{p,q}(Z)}
 \in1+Z\pos[[Z]].
\end{equation}
For every $m\in\Z^2$, the oriented action on $X^m$ is positive:
use $\Ad_{f_{p,q}}$ if $\det(pu+qw,m)\geq0$ and
$\Ad_{f_{p,q}^{-1}}$ if this determinant is nonpositive.
\end{proposition}
\begin{proof}
Take the star lattice with source basis vectors $U_0,\ldots,U_{k-1}$
and sink vector $W$. For the Hall torus use
$B(d,e)=\chi(e,d)-\chi(d,e)$, so $B(U_i,W)=b$.
Thus the torus product is $z^dz^e=t^{B(d,e)}z^{d+e}$.
The representations are maps from sources to sink; this is
equivalently the opposite quiver in the right-module convention of
\cite[\S6]{DM}.

Let $A_H$ and $B_H$ be the source-only and sink-only Hall series.
Harder--Narasimhan factorization with source weight one and sink weight
zero, and then with the weights reversed, gives
\begin{equation}\label{eq:hall}
 B_HA_H=A_H\left(\prod_{q/p\,\mathrm{increasing}}H_{p,q}\right)B_H.
\end{equation}
The slope inequality is exactly~\eqref{eq:stability} at dimensions
$(np,nq)$. At the reversed stability every mixed representation has its
sink subrepresentation as a destabilizing subobject. By the integrated
Hall factorization and BPS description in
\cite[Theorems~6.3, 6.6 and Proposition~6.14]{DM},
$H_{p,q}$ is the plethystic exponential of the centered BPS characters
$\Omega_a(t)$ from~\eqref{eq:bpspolynomial}. Hall multiplication here
encodes extensions; the Harder--Narasimhan filtration groups their
successive semistable factors by slope. Integration replaces these
classes by series in the torus with Euler-form multiplication. The
plethystic exponential packages the symmetric powers of the primitive
spaces; its signed quantum-dilogarithm expression is given in
\eqref{eq:starfactor}, fixing the convention used below.

Pass to the opposite torus $B\mapsto-B$. The map $\iota$ fixing
normalized monomials is an anti-isomorphism; therefore
$g\mapsto\iota(g^{-1})$ is a group homomorphism. It sends the source
and sink Hall factors to
$\prod_i\Psi_t(z^{U_i})$ and $\Psi_t(z^W)$, respectively, and
sends~\eqref{eq:hall} to the order in~\eqref{eq:rayfactor}.
If $\beta_{n,a,j}=[t^j]\Omega_a(t)$, its mixed factor is
\begin{equation}\label{eq:starfactor}
 \prod_{n,a,j}
 \Psi_t\left((-1)^jt^jz^{\sum_i a_iU_i+nqW}\right)
 ^{(-1)^j\beta_{n,a,j}}.
\end{equation}
This follows from the plethystic convention of
\cite[(26), (28), (30)]{DM}. In particular the odd-degree factor is
$\Psi_t(-t^jT)^{-1}$.

The opposite star torus maps to the original torus by
\begin{equation}\label{eq:projection}
 z^{U_i}\longmapsto v^{b(k-1-2i)}X^u,
 \qquad z^W\longmapsto X^w.
\end{equation}
Pairings agree, since $t^{-b}=v^{bc}=v^{-\det(u,w)}$.
Each degree fiber is a finite set of weak compositions. Comparing formal
logarithms, the identity
\[
 \sum_{i=0}^{k-1}
 \frac{v^{bh(k-1-2i)}}{v^{-ch}-v^{ch}}
 =\frac1{v^{-bh}-v^{bh}}
\]
shows that the incoming source product merges to $A_0$. The sink factor
maps to $B_0$. Uniqueness of ordered factorization identifies the image
of~\eqref{eq:starfactor} with $f_{p,q}$.

The BPS parity assertion of \cite[Lemma~6.13]{DM} says that its
degrees have parity $\chi(d,d)+1$. Since $\sum_i a_i=np$,
\begin{equation}\label{eq:parity}
 \chi(d,d)+1\equiv1+n(p+q-bpq)\pmod2,
\end{equation}
independently of the color composition. Set $x=v^{-b}$. The projected
argument in~\eqref{eq:starfactor} is
$(-1)^jx^{kj-\sum_i(k-1-2i)a_i}Z^n$.
Color reversal changes the minus sign to a plus sign and leaves the
moduli problem unchanged. Theorem~\ref{thm:colors},
\eqref{eq:parity} and Proposition~\ref{prop:difference} now give
\eqref{eq:minimal}.

Finally, for $D=pu+qw=(-bp,cq)$,
\[
 \det(D,\Z^2)=\gcd(bp,cq)\Z=b\gcd(p,k)\Z.
\]
If $a=\det(D,m)=Lbg\geq0$, then
\[
 \frac{f_{p,q}(v^{-2a}Z)}{f_{p,q}(Z)}
 =\prod_{j=0}^{L-1}
 \frac{f_{p,q}(v^{-2(j+1)bg}Z)}{f_{p,q}(v^{-2jbg}Z)}
\]
is positive by~\eqref{eq:minimal}. The identity
$f_{p,q}(Z)X^m=X^mf_{p,q}(v^{-2a}Z)$ gives the adjoint action. Normalizing
$X^mZ^n$ multiplies its coefficient by $v^{na}$ and preserves positivity.
For $a=-Lbg\leq0$, use
$f_{p,q}(Z)/f_{p,q}(v^{2Lbg}Z)$, the same positive quotient after a monomial shift
of its argument. Zero pairing gives the identity action.
\end{proof}

\section{Identification with quantum greedy elements}
\subsection{A three-torus criterion}
Put $R=\Z[v^{\pm1}]$ and let $\mathcal T_{ij}$ be the quantum
Laurent torus on $X_i,X_j$. In particular,
\[
 X_0=X_2^{-1}(v^bX_1^b+1),\qquad
 X_3=(v^cX_2^c+1)X_1^{-1}.
\]
Denote the classical greedy element by $x[A,N]$.

\begin{lemma}\label{lem:recognition}
Suppose $F\in\mathcal T_{12}\cap\mathcal T_{01}\cap\mathcal T_{23}$,
its normalized $\mathcal T_{12}$ coefficients belong to $\pos$, and
$F|_{v=1}=x[A,N]$. Then $F=v^sG_{bc}(A,N)$ for some $s\in\Z$.
\end{lemma}
\begin{proof}
A positive Laurent polynomial vanishes at $v=1$ only if it is zero.
Hence the support of $F$ is exactly that of $x[A,N]$. Its coefficient
at $X^{(-A,-N)}$ specializes to one and is positive, so it is $v^s$.
Thus $Y=v^{-s}F$ is pointed at $(A,N)$, with leading coefficient one,
and has the classical greedy pointed support.

The region used in the support characterization of
\cite[Definition~3.1]{LLRZ} is the following subset of the nonnegative
$(p,q)$-plane. It is the origin if $A,N\leq0$; the segment
$q=0$, $0\leq p\leq N$ if $A\leq0<N$; and the segment
$p=0$, $0\leq q\leq A$ if $N\leq0<A$. For $0<bN\leq A$
it is $p\leq N$, $q\leq A-bp$; for $0<cA\leq N$ it is
$q\leq A$, $p\leq N-cq$. In the remaining positive case it is
\[
 \left\{p<A/b,\quad q+\left(b-\frac{bN}{cA}\right)p<A\right\}
 \ \cup\
 \left\{q<N/c,\quad p+\left(c-\frac{cA}{bN}\right)q<N\right\}
 \ \cup\{(0,A),(N,0)\}.
\]
The strict boundary conditions are retained by equality of supports.

Write $Y=\sum d(p,q)X^{(-A+bp,-N+cq)}$ and
$B_{m,u}(T)=\sum_i\qbin mi uT^i$ for $m\geq0$.
Adjacent Laurentness gives
\[
 B_{N-cq,v^b}(T)\mid\sum_p d(p,q)T^p\quad(0\leq q<N/c),
\]
and the analogous column divisibility by $B_{A-bp,v^c}$ for
$0\leq p<A/b$. These are the divisibility conditions of
\cite[Lemma~2.2]{LLRZ}. The horizontal parameter is $v^b$:
the exchange relation gives, for $m\geq0$ and $r\in\Z$,
\[
 v^{mr}X_0^mX_1^r
   =\sum_{i=0}^m\qbin mi{v^b}X^{(r+bi,-m)}.
\]
This also fixes the parameter typo in the first divisor of
\cite[equation~(2.1)]{LLRZ}, in the cited version.
The support and divisibility characterization
\cite[Proposition~3.7]{LLRZ} gives $Y=G_{bc}(A,N)$.
\end{proof}

\begin{lemma}\label{lem:positive-specialization}
Let $L$ be a lattice and $C\subset L$ a finitely generated monoid
contained in a pointed rational cone. Consider the quantum-torus
completion whose supports lie in a finite union of translates of $C$,
with the topology given by an integral degree strictly positive on
$C\setminus\{0\}$. Suppose
\[
 F=\sum_{m\in L}c_m(v)X^m,\qquad c_m(v)\in\pos,
\]
and its coefficientwise specialization is a finite Laurent polynomial.
Then $F$ has finite support and
$\supp(F)=\supp(F|_{v=1})$.
\end{lemma}
\begin{proof}
The pointed-cone condition makes each bounded-degree truncation finite.
Every $c_m(v)$ is a finite Laurent polynomial, so its value at $v=1$
is defined coefficientwise. Nonnegativity gives
$c_m(1)=0$ if and only if $c_m(v)=0$. Thus the two supports agree,
and finiteness of the specialized support proves the assertion.
\end{proof}

\subsection{The scattering diagram on
  \texorpdfstring{$\mathbb Z^2$}{Z²}}
We use a scattering diagram to construct the lift required by
Lemma~\ref{lem:recognition}. In this two-dimensional setting, a wall
consists of a line or ray together with a formal torus automorphism.
Crossing it applies that automorphism or its inverse according to the
orientation. Consistency means that the ordered product along a path
depends only on its endpoints, away from the walls and their
intersections. Both diagrams and their products are interpreted at each
finite degree before passing to the completion.

A broken line carries a monomial on each segment. At a crossing, one
chooses a term of the wall action as the monomial of the next segment.
The theta series sums the final monomials over broken lines with a
fixed initial exponent and endpoint. Proposition~\ref{prop:wall}
gives positive contributions to this series. The formal construction and
transport statements we use are \cite[Proposition~3.1 and
Theorem~3.2]{DM}.

Take $L=L_0=\Z^2$, skew form $\omega=\det$, parameter $t_0=v^{-1}$,
and cone $\sigma=\R_{\geq0}(-1,0)+\R_{\geq0}(0,1)$.
The incoming axes carry $A_0$ and $B_0$. This is within the general
quantum Lie-algebra setup of \cite[\S2.2--2.3]{DM}: the lattice
indices of $hu$ and $hw$
are $hb$ and $hc$, respectively, so the denominators in
\eqref{eq:psi} are exactly those of its generators
$z^a/(t_0^{|a|}-t_0^{-|a|})$.

The consistent completion \cite[Theorem~2.13]{DM} adds outgoing
mixed rays in the fourth quadrant. Its degrees lie in
$\mathcal P=\N u+\N w$; Lie operations add degrees and cannot create
new pure-axis corrections. A counterclockwise loop gives
\[
 A_0\left(\prod_{q/p\,\mathrm{increasing}} f_{p,q}\right)
 B_0A_0^{-1}B_0^{-1}=1.
\]
Indeed the first three crossings give $B_0^{-1},A_0^{-1},B_0$ in
time order, later crossings act on the left, and the mixed rays are
crossed in decreasing $q/p$ order. Thus its ray factors are
precisely~\eqref{eq:rayfactor}.

By Proposition~\ref{prop:wall}, every outgoing oriented action on every
full-lattice monomial is positive. The same is true for the incoming
actions: their determinant pairings are multiples of $b$ or $c$, and
$\Psi_s(s^2Z)/\Psi_s(Z)=1+sZ$ expresses the appropriate oriented
action as a finite product of positive binomials.

Let $\vartheta_{h,Q}$ be the broken-line theta series with exponent
$h\in\Z^2$ and generic endpoint $Q$. A segment with exponent $m$ has velocity $-m$,
so the bend uses exactly the sign $\det(D,m)$ considered above.
Every coefficient of $\vartheta_{h,Q}$ is a finite positive Laurent
polynomial. To check finiteness, fix an increment $pu+qw$.
There are at most $p+q$ nontrivial bends, finitely many relevant walls
in that truncation, and finitely many increment sequences. For fixed
sequences and a generic endpoint, backward tracing determines at most
one broken line. The case $h=0$ gives the constant series one.

The classical-limit homomorphism of \cite[\S2.2.4]{DM} carries the
incoming actions to those of $1+x_1^{-b}$ and $1+x_2^c$ with primitive
normal exponents. Uniqueness identifies the specialized completion with
the ordinary rank-two diagram. Finite degreewise sums commute with this
specialization. By \cite[Theorem~4.3, Remark~4.5 and Theorem~5.1]{CGMMRSW},
at a first-quadrant endpoint the ordinary theta label $h$ corresponds to
the negative denominator vector, or pointed leading exponent,
\[
 T(h_1,h_2)=(h_1+b\min(h_2,0),h_2).
\]
For $h=(-A+b[N]_+,-N)$, this is $(-A,-N)$. Hence
\[
 F=\vartheta_{h,Q},\qquad F|_{v=1}=x[A,N].
\]
The support of $F$ lies in $h+\mathcal P$, and $\mathcal P$ lies
in the pointed cone $\sigma$. Lemma~\ref{lem:positive-specialization}
therefore gives $F\in\mathcal T_{12}$.

\subsection{Laurentness in adjacent clusters}
Let $\Phi_u=\Ad_{A_0}$ and $\Phi_w=\Ad_{B_0}$. Introduce the
monomial lattice maps
\[
 \tau_L=\begin{pmatrix}b&1\\-1&0\end{pmatrix},\qquad
 \tau_R=\begin{pmatrix}0&-1\\1&0\end{pmatrix}.
\]
Both have determinant one. Direct use of~\eqref{eq:torus} gives
\begin{align*}
 \Phi_u(X^{(b,-1)})&=X^{(b,-1)}+X^{(0,-1)}=X_0,&
 \Phi_u(X^{(1,0)})&=X_1,\\
 \Phi_w(X^{(0,1)})&=X_2,&
 \Phi_w(X^{(-1,0)})&=X^{(-1,0)}+X^{(-1,c)}=X_3.
\end{align*}
For example, $X^{(b,-1)}X^u=v^bX^{(0,-1)}$, canceling the
$v^{-b}$ in the incoming binomial. Thus the ordered-chart
embeddings are $\Phi_u\tau_L$ and $\Phi_w\tau_R$.

Fix a bound on $\mathcal P$-degree. Near a nonzero point of the positive
horizontal axis choose a short downward path meeting no other wall of
that finite truncation. Its upper endpoint and the endpoint defining
$F$ lie in the wall-free first quadrant, so their theta expansions
agree by \cite[Theorem~3.2]{DM}. The crossing sign is negative.
Theta transport identifies the lower-endpoint series with
$\Phi_u^{-1}(F)$ to the chosen degree. Since every endpoint theta is
positive, this proves coefficientwise positivity of $\Phi_u^{-1}(F)$.
For two degree bounds, the resulting expressions are truncations of
the same fixed series $\Phi_u^{-1}(F)$. Thus, for any specified
coefficient, a sufficiently large bound proves its positivity
independently of the endpoints chosen. A short leftward crossing of the positive vertical axis proves
the same for $\Phi_w^{-1}(F)$.

Applying $\tau_L^{-1}$ or $\tau_R^{-1}$ relabels monomials without
changing coefficients. The specializations of these two series are
the adjacent classical Laurent expansions of $x[A,N]$. More explicitly,
embed the adjacent torus in the completion for
$\tau_L^{-1}(\mathcal P)$ or $\tau_R^{-1}(\mathcal P)$.
Each exchange expression is a monomial times a series with constant
term one; its inverse has a unique expansion in that pointed completion.
Thus the specialized formal inverse agrees with the finite classical
Laurent expression. The inverse wall actions add degrees in
$\mathcal P$, and the lattice maps send this monoid into the pointed
monoids $\tau_L^{-1}(\mathcal P)$ and $\tau_R^{-1}(\mathcal P)$.
Applying Lemma~\ref{lem:positive-specialization} in these completions
proves finite support of both inverse expressions.
The two inverse expressions are adjacent Laurent polynomials
whose embeddings equal $F$. We have proved the three-torus hypotheses
of Lemma~\ref{lem:recognition}. It follows that $G_{bc}(A,N)$ is
positive in the initial cluster for every integer label when $b\mid c$.

\subsection{Symmetries and universal positivity}
\begin{proof}[Proof of Theorem~\ref{thm:main}]
The recurrence is invariant under
\[
 (b,c,A,N,p,q)\longmapsto(c,b,N,A,q,p).
\]
Induction on $p+q$ therefore gives
$e_{bc;A,N}(p,q)=e_{cb;N,A}(q,p)$, including nonpositive labels.
The case $c\mid b$ follows.

Consider the semilinear reflections
$\sigma_\ell(X_j)=X_{2\ell-j}$, $\sigma_\ell(v)=v^{-1}$.
The greedy reflection formulas of \cite[Proposition~5.1,
Theorem~5.8 and the proof of Theorem~1.9]{LLRZ} give, for $(A,N)\in\Z^2$,
\[
 \sigma_1G_{bc}(A,N)=G_{bc}(A,c[A]_+-N),\qquad
 \sigma_2G_{bc}(A,N)=G_{bc}(b[N]_+-A,N).
\]
These reflections generate
translations by two and act transitively on unordered adjacent pairs.
After transporting a positive initial expansion, restoring ascending
order changes no sign, since
\[
 v^{-rs}X_{j+1}^rX_j^s=v^{rs}X_j^sX_{j+1}^r.
\]
Inverting $v$ preserves nonnegative coefficients. The transported
expansion is therefore positive in every ordered cluster, as required.

\end{proof}

\bigskip
\begin{flushleft}
\small
\textsc{Zhili College, Tsinghua University, Beijing 100084, China}\\[2pt]
\textit{Email address:}\enspace
\href{mailto:tangqy24@mails.tsinghua.edu.cn}{\texttt{tangqy24@mails.tsinghua.edu.cn}}
\end{flushleft}
\end{document}